\documentclass[12pt,reqno]{amsart}
\usepackage{fullpage}
\usepackage[small,nohug,heads=vee]{diagrams}
\diagramstyle[labelstyle=\scriptstyle]
\usepackage{amsmath}
\usepackage{epsfig}
\usepackage{amsmath,amsfonts,amsthm,amscd, amssymb}
\usepackage{latexsym}
\usepackage{enumerate}
\usepackage{mypack}
\numberwithin{equation}{thm}
\newcommand{\co}{:}

\title{Homotopy colimits of classifying spaces of abelian subgroups of a finite group}
\date{}
\address{Department of Mathematics, University of British Columbia, Vancouver BC V6T 1Z2, Canada}
\email{okay@math.ubc.ca}
\author{C\.{I}han Okay}

\begin{document}
  \maketitle 

\begin{abstract}
The  classifying space $BG$ of a topological group $G$ can be filtered by a sequence of subspaces $B(q,G)$, $q\geq 2$, using the descending central series of free groups. 
If $G$ is finite, describing  them   as homotopy colimits is  convenient when applying homotopy theoretic methods. In this paper we  introduce natural subspaces $B(q,G)_p\subset B(q,G)$ defined for a fixed prime $p$.  We show that $B(q,G)$ is stably homotopy equivalent to a wedge of $B(q,G)_p$ as $p$ runs over the primes dividing the order of $G$.  
 Colimits of abelian groups play an important role in understanding the homotopy type of  these spaces. Extraspecial $2$--groups are key examples, for which  these colimits turn out to be finite. We prove that for extraspecial $2$--groups  of order $2^{2n+1}$, $n\geq 2$, $B(2,G)$ does not have the homotopy type of a  $K(\pi,1)$ space thus answering in a negative way a question posted in \cite{adem1}. For a finite group $G$, we compute the complex $K$--theory of $B(2,G)$ modulo torsion.
\end{abstract}
\maketitle
\section{Introduction}
In \cite{adem1} a natural filtration of the classifying space $BG$ of a topological group $G$ is introduced:
\[
B(2,G)\subset \cdots\subset B(q,G)\subset \cdots\subset B(\infty,G)=BG.
\]
For a fixed $q\geq 2$, let $\Gamma^q(F_n)$ denote the $q$--th stage of the descending central series of the free group on $n$ generators. Then $B(q,G)$ is the geometric realization of the simplicial space whose $n$--simplices are the spaces of homomorphisms $\Hom(F_n/\Gamma^q(F_n),G)$. 
In this article we study homotopy--theoretic properties of these spaces.

Let $G$ be a finite group and $p$ a prime dividing the order of $G$. Consider the free pro--$p$ group $P_n$, the pro--$p$ completion of the free group $F_n$. As noted in \cite{adem1}, the geometric realization of the simplicial set $\catn \mapsto \Hom(P_n/\Gamma^q(P_n),G)$ gives a natural subspace of $B(q,G)$, and is denoted by $B(q,G)_p$. We prove that there is a stable homotopy equivalence:
\Thm{\label{int_stable} Suppose that $G$ is a finite group. There is a natural weak equivalence
\[
\bigvee_{p||G|}\Sigma B(q,G)_p\rightarrow \Sigma B(q,G)\;\; \text{ for all }q\geq2
\]
induced by the inclusions $B(q,G)_p\rightarrow B(q,G)$.
}

Let $\nN(q,G)$ denote the collection of subgroups of nilpotency class at most $q$ and $G(q)= \colim{\nN(q,G)}A$. The key observation in \cite{adem1} is the following fibration 
\[
 \hocolim{\nN(q,G)} G(q)/A \rightarrow B(q,G)\rightarrow BG(q),
\]
which can be constructed using the (homotopy) colimit description of $B(q,G)$.
This raises the following question of whether they are actually homotopy equivalent posted in \cite{adem1}. 
\begin{que}\cite[page 15]{adem1}
If $G$ is a finite group,  are the spaces $B(q,G)$ Eilenberg--Mac Lane spaces of type $K(\pi,1)$? 
\end{que}
 One of the  objectives of this paper is to show the existence of a certain class of groups for which  $B(2,G)$ does not  have the homotopy type of a $K(\pi,1)$ space. The following theorem is used to show that extraspecial $2$--groups are examples of such groups. (At a given order there are two types up to isomorphism.) 
\begin{thm}\label{int_fundamental}
Let $G_n$ denote an extraspecial $2$--group of order $2^{2n+1}$ then 
\[ \pi_1(B(2,G_n))\cong \colim{\nN(2,G_n)}A\cong G_n\times \ZZ/2 \;\; \text{ for } n\geq2. \]
\end{thm}  
Indeed for  the central products $D_8\circ D_8$ and $D_8\circ Q_8$, that is $n=2$, the higher homotopy groups are given by
\[
\pi_i(B(2,G_2))\cong \pi_i(\bigvee^{151}S^2)\; \text{ for } i>1.
\]

Another natural question is to compute the complex $K$--theory of a homotopy colimit. The main tool to study a  representable generalized cohomology theory of a homotopy colimit is the Bousfield--Kan spectral sequence whose $E_2$--term consists of the derived functors of the inverse limit functor.  We address this problem for the homotopy colimit of classifying spaces of abelian subgroups of a finite group, that is, for $B(2,G)$. 
\begin{thm}\label{int_Ktheory}
 There is an isomorphism
\[
\QQ \otimes K^i(B(2,G))\cong
\left\lbrace
\begin{array}{ll}
\QQ \oplus \bigoplus_{p\mid |G|} \QQ_p^{n_p} & \text{ if $i=0$, }\\
0 & \text{ if $i=1$, }
\end{array} \right.
\]
where $n_p$ is the number of (non-identity) elements of order a power of $p$ in $G$.
\end{thm}
Torsion groups can appear in $K^1(B(2,G))$. An explicit example is the case of $G_2$, there is an isomorphism
\[
K^i(B(2,G_2)) 
\cong \left\{
\begin{array}{ll}
 \ZZ \oplus \ZZ_2 ^{31}&  \text{ if $i=0$, }  \\
(\ZZ /2)^9 &\text{ if $i=1$}. 
\end{array} \right.
\]

The organization of this paper is as follows. In \S \ref{sec:pre} we introduce the spaces $B(q,G)$ and $B(q,G)_p$ for discrete groups and list some of their basic properties. In \S  \ref{sec:finite} we assume $G$ is finite, and describe these spaces as homotopy colimits. We prove a stable decomposition of $B(q,G)$ (Theorem \ref{int_stable}). 
In \S \ref{sec:higherlimits} we turn to the higher limits as a preparation for the complex $K$--theory of $B(2,G)$. Under some assumptions, we prove a vanishing result of certain higher limits (Theorem \ref{torsion}).  The main result of \S \ref{sec:ktheory} is the computation of $\QQ\otimes K^*(B(2,G))$  as given in Theorem \ref{int_Ktheory}.
\S \ref{sec:colimits} contains a general observation on colimits of abelian groups (Theorem \ref{rank2}). As an application, in \S \ref{sec:colimext} we compute the colimit of abelian subgroups of an extraspecial $2$--group (Theorem \ref{int_fundamental}).
Key examples  which illustrate the main features discussed throughout the paper are given in \S \ref{sec:examples}. Basic properties of homotopy colimits are explained in the appendix, \S \ref{sec:app}.

The author would like to thank A. Adem, for his supervision; J. Smith, for pointing out Proposition \ref{cech}; F. Cohen and J. M. G\'{o}mez, for commenting on an early version of this paper.

\section{Filtrations of classifying spaces}\label{sec:filtration}

\subsection{Preliminaries}\label{sec:pre}

Let $\catDel$ be the category whose objects are  finite non-empty totally ordered sets  $\catn$, $n\geq 0$, with $n+1$ elements
\[
0\rightarrow 1\rightarrow 2\rightarrow \cdots \rightarrow n
\]
 and whose morphisms $\theta \co \catm\rightarrow \catn$ are order preserving set maps or alternatively functors. 
The \textit{nerve} of a small category $\catC$ is the simplicial set
\[
B\catC_n = \Hom _{\Cat}(\catn,\catC).
\] 
For instance,  the geometric realization of the nerve $BG_\n$ of a discrete group $G$ (regarded as a category with one object) is  the classifying space $BG$. An equivalent way of describing the classifying space is as follows. The assignment 
\[
\catn=(0\stackrel{e_1}{\rightarrow}1\stackrel{e_2}{\rightarrow}\cdots\stackrel{e_n}{\rightarrow} n)
\mapsto F_n 
\]
where  $F_n$  is the free group generated by  $\lbrace e_1,e_2,...,e_n \rbrace $, defines a faithful functor $\catDel \rightarrow \catGrp$ injective on objects. Then the classifying space of $G$ is isomorphic to the simplicial set
\[
\begin{array}{rcl}
BG_\n\co \catDel^{op} &\rightarrow &\catSet \\
\catn &\mapsto& \Hom_\catGrp(F_n,G).
\end{array}
\]

A generalization of this construction to certain quotients of  free groups is studied in \cite{adem1}. We recall the constructions and describe  some alternative versions.

\Def{\rm{ For a group $Q$  define a chain of groups inductively: $\Gamma ^1(Q)=Q$, $ \Gamma ^{q+1} (Q)= [ \Gamma ^{q} (Q) ,Q]   $.  The \textit{descending central series} of $Q$ is the normal series
$$
1\subset\cdots\subset \Gamma ^{q+1} (Q) \subset  \Gamma ^{q} (Q) \subset \cdots \subset  \Gamma ^{2} (Q) \subset \Gamma ^{1} (Q)=Q.
$$
}}

Now take $Q$ to be the free group $F_n$. For $q>0$, the natural maps $F_n\rightarrow F_n/\Gamma^q(F_n)$ induce inclusions of sets
\[
\Hom_\catGrp(F_n/\Gamma^q(F_n),G)\subset \Hom_\catGrp(F_n,G).
\]
Furthermore, the simplicial structure of $BG_\n$ induces a simplicial structure on the collection of these subsets.

\Def{\rm{Let $G$ be a discrete group.  We define a sequence of simplicial sets $B_\n(q,G)$, $q\geq 2$, by the assignment 
\[
\catn \mapsto \Hom_\catGrp(F_n/\Gamma^q(F_n),G)
\]
and denote the geometric realization $|B(q,G)_\n|$ by $B(q,G)$.
Define $B_\n(\infty ,G)=BG_\n$ and $B(\infty,G)=BG$ by convention.
}}

As in \cite{adem1} this gives rise to a filtration of the classifying space of $G$
\[
B(2,G)\subset B(3,G)\subset \cdots\subset B(q,G)\subset B(q+1,G) \subset \cdots\subset B(\infty,G)=BG.
\]
\subsection{$p$--local version}
An alternative construction is obtained by replacing  the free group $F_n$ by its pro--$p$ completion, namely the free pro--$p$ group $P_n$, see \cite{D-S-M-S} for their basic properties. It can be defined using the \textit{$p$-descending central series} of the free group 
\[
P_n=\ilim{}{} F_n/\Gamma ^q_p(F_n)
\]
where $\Gamma^q_p(F_n)$ are defined recursively: 
\[ \Gamma^1_p(F_n)=F_n \text{ and } \Gamma^{q+1}_p(F_n)=[\Gamma^q_p(F_n),F_n](\Gamma^q_p(F_n))^p.\]
The group $P_n$ contains $F_n$ as a dense subgroup. Observe that each quotient $F_n/\Gamma ^q_p(F_n)$  is a finite $p$--group.

Let $\catTGrp$ denote the category of topological groups.
We consider continuous group homomorphisms $\phi\in \Hom_\catTGrp(P_n,G)$  where $G$ is a discrete group. The image of $\phi$ is a finite $p$--subgroup of $G$. This follows from the fact that $P_n$ is compact and if $K$ is a subgroup of $P_n$ of finite index then $|P_n:K|$ is a power of $p$ by \cite[Lemma 1.18]{D-S-M-S}.
Then there exists  $q\geq 1$ such that $\phi$ factors as
\[
\begin{diagram}
P_n & \rTo^\phi & G\\
\dTo & \ruTo \\
F_n/\Gamma^q_p(F_n).
\end{diagram}
\]
This implies that $\Hom_\catTGrp(P_n,G)\subset \Hom_\catGrp(F_n,G)$ and also there is a simplicial structure induced from $BG_\n$.

We remark that a theorem of Serre on topological groups says that any (abstract) group homomorphism from a finitely generated pro--$p$ group to a finite group is continuous \cite[Theorem 1.17]{D-S-M-S}. This implies that when $G$ is finite any (abstract) homomorphism $\phi\co P_n\rightarrow G$ is continuous. More generally, as a consequence of \cite[Theorem 1.13]{N-S} if every subgroup of $G$ is finitely generated then  the image of an abstract homomorphism $\phi$ is a finite group, hence in this case $\phi$  is also continuous.

\Def{\rm{ Let $p$ be a prime integer and $G$ a discrete group. We define a sequence of simplicial sets $B_\n(q,G)_p$, $q\geq 2$, by 
\[
\catn \mapsto \Hom_\catTGrp(P_n/\Gamma^q(P_n),G).
\]
Define $B_\n(\infty ,G)_p $ to be the simplicial  set $\catn \mapsto \Hom_\catTGrp(P_n,G)$ and set $B(\infty,G)_p=|B_\n(\infty,G)_p|$.

}}

Let us examine the simplicial set $B_\n(\infty,G)_p$. 
 There is a natural map
\[
\theta \co \dlim{q} \Hom (F_n/\Gamma ^q_p(F_n),G)\rightarrow \Hom (P_n, G)
\]
induced by the projections $P_n\rightarrow F_n/\Gamma ^q_p(F_n)$. This map is injective since it is induced by injective maps and surjective since $\phi\co P_n\rightarrow G$ factors through $F_n/\Gamma^q_p(F_n)\rightarrow G$ for some $q\geq 1$, as observed above. Thus $\theta $ is a bijection of sets. 
Colimits in the category of simplicial sets can be constructed  dimension-wise and the geometric realization functor commutes with colimits. Therefore, for a discrete group $G$, we note that there is a homeomorphism
\[
\dlim{q} B(q,G,p)\rightarrow B(\infty,G)_p
\]
induced by $\theta$, where $B(q,G,p)$ is the geometric realization of the simplicial set defined by $B_n(q,G,p)=\Hom_\catGrp(F_n/\Gamma^q_p(F_n),G)$.

\section{The case of finite groups}\label{sec:finite}
 We restrict our attention to finite groups. Note that the following collections of subgroups have an initial object (the trivial subgroup), and they are closed under taking subgroups and the conjugation action of $G$. Define a collection of subgroups for $q\geq 2$ 
\[
\nN(q,G)=\set{H\subset G|\;\Gamma^q(H)=1}.
\]
These are  nilpotent subgroups of $G$ of class at most $q$. Observe that $\nN(2,G)$ is the collection of abelian subgroups of $G$.

Denote the poset of $p$--subgroups of $G$ by $\sS_p(G)$ (including the trivial subgroup) and set 
\[ \nN(q,G)_p=\sS_p(G)\cap \nN(q,G) \]
when $q=2$ this is the collection of abelian $p$--subgroups of $G$.

\subsection{Homotopy colimits}
In \cite{adem1} it is observed that there is a homeomorphism
\[
\colim{\nN(q,G)} BA\rightarrow B(q,G)
\]
 and furthermore the natural map 
\[
\hocolim{\nN(q,G)}BA\rightarrow \colim{\nN(q,G)}BA
\]
turns out to be a weak equivalence. A similar statement holds for the spaces $B(q,G)_p$.

\Pro{\label{hocolim}Suppose that  $G$ is a finite group.  Then there is a homeomorphism
\[ B(q,G)_p\cong \colim{\nN(q,G)_p}BP\]
and the natural map 
\[
\hocolim{\nN(q,G)_p}BP\rightarrow \colim{\nN(q,G)_p}BP
\]
is a weak equivalence.
}
\Proof{ There is a natural bijection of sets
\[
\colim{\nN(q,G)_p}\Hom_\catGrp(F_n,P)\rightarrow  \Hom_\catGrp(P_n/\Gamma^q(P_n),G).
\]
The image of a morphism $P_n/\Gamma^q(P_n)\rightarrow G$ is a $p$--group of nilpotency class at most $q$. Conversely, if $P\subset G$ is a $p$--group of nilpotency class $q$ then $F_n\rightarrow P$ induces a map $P_n\rightarrow P$ by taking the pro--$p$ completions, and factors through the quotient $P_n/\Gamma^q(P_n)$. 
This induces the desired homeomorphism $B(q,G)_p\cong \colim{\nN(q,G)_p}BP$. The natural map from the homotopy colimit to the ordinary colimit is a weak equivalence since the diagram of spaces is free (\S \ref{sec:app}). 
}

\Rem{\label{cofinal}\rm{For the homotopy colimits considered in this paper, it is sufficient to consider the sub-poset determined by the intersections of the maximal objects of the relevant poset. More precisely, let $M_1,M_2,...,M_k$ denote the maximal groups in the poset $\nN(q,G)$ and  $\mM(q,G)=\set{\cap_J M_i|\; J\subset \set{1,2,...,k}}$. The inclusion map
\[
\mM(q,G)\rightarrow \nN(q,G)
\] 
is \textit{right cofinal} and hence induces a weak equivalence  (\cite[Proposition 3.10]{wzz})
\[
\hocolim{\mM(q,G)}BM\rightarrow \hocolim{\nN(q,G)}BA.
\]
A similar observation holds for $\nN(q,G)_p$.
}}

\subsection{A stable decomposition of $B(q,G)$}

Recall that a finite nilpotent group $N$ is a direct product of its Sylow $p$--subgroups
\[
N\cong \prod_{p||N|}N_{(p)}.
\]
For a prime $p$ dividing the order of the group $G$, and $N\in \nN(q,G)$ define
\[
\pi_p(N)=\left\lbrace \begin{array}{ll}
N_{(p)} & \text{ if } p||N| \\
1 & \text{ otherwise }.
\end{array}\right.
\]

\Lem{\label{pushdown}The inclusion map $\iota_p\co\nN(q,G)_p\rightarrow \nN(q,G)$ of posets induces a weak equivalence
\[
\hocolim{\nN(q,G)_p} BP\rightarrow \hocolim{\nN(q,G)}B\pi_p(A).
\]
}
\Proof{ By freeness (\S 9) it is enough to consider ordinary colimits.  Let $B\co\nN(q,G)_p\rightarrow \catTop$ denote the functor $P\mapsto BP$. The map $\pi_p\co\nN(q,G)\rightarrow \nN(q,G)_p$ defined by $A\mapsto \pi_p(A)$ induces an inverse to the map induced by $\iota_p$ on the colimit
\[
\colim{\nN(q,G)_p} BP\rightarrow \colim{\nN(q,G)}B\pi_p(A).
\]
This follows from the fact that $ \pi_p\circ \iota_p$ is the identity on $\nN(q,G)_p$. 
}

\Thm{\label{decomposition} Suppose that $G$ is a finite group. There is a natural weak equivalence
\[
\bigvee_{p||G|}\Sigma B(q,G)_p\rightarrow \Sigma B(q,G)\;\; \text{ for all }q\geq2
\]
induced by the inclusions $B(q,G)_p\rightarrow B(q,G)$.
}
\Proof{ First note that the nerve of the poset $\nN(q,G)$ is contractible since the trivial subgroup is an initial object. Note that for $A\in \nN(q,G)$ each $BA$ is pointed via the inclusion $B1\rightarrow BA$.

 Let $\jJ$ denote the pushout category $0\leftarrow 01\rightarrow 1$. Define a functor $F$ from the product category $\jJ\times \nN(q,G)$ to $\catTop$ by
\[
F((0,A))=\pt,\; F((01,A))=BA\; \text{ and } F((1,A))=\pt.
\]
 Commutativity of homotopy colimits implies that
\[
\hocolim{\jJ}\hocolim{\nN(q,G)}F\cong \hocolim{\jJ\times \nN(q,G)} F \cong \hocolim{\nN(q,G)}\hocolim{\jJ}F.
\]
The nerve of $\nN(q,G)$ is contractible, in which case we can identify the homotopy colimit over $\jJ$ as the  suspension and conclude that the map
\begin{equation}\label{eq:suspension}
\Sigma (\hocolim{\nN(q,G)}BA)\rightarrow \hocolim{\nN(q,G)}\Sigma(BA)
\end{equation}
is a homeomorphism.
The natural map induced by the suspension of the inclusions $BA_{(p)}\rightarrow BA$
\[
\bigvee_{p||A|}\Sigma BA_{(p)} \rightarrow \Sigma BA 
\]
is a weak equivalence.
This follows from the splitting of suspension of products: 
\[ \Sigma(BA_{(p)}\times BA_{(q)} ) \simeq \Sigma (BA_{(p)})\vee \Sigma(BA_{(q)})\vee \Sigma(BA_{(p)}\wedge BA_{(q)}) \] 
and from the equivalence $\Sigma(BA_{(p)}\wedge BA_{(q)})\simeq \pt$ when $p$ and $q$ are coprime. The latter follows from the K\"{u}nneth theorem and the Hurewicz theorem by considering  the homology isomorphism induced by the inclusion $BA_{(p)}\vee BA_{(q)}\rightarrow BA_{(p)}\times BA_{(q)}$.

 Invariance of homotopy colimits under natural transformations which induce a weak equivalence on each object, the homeomorphism (\ref{eq:suspension}) and Lemma \ref{pushdown} give the following  weak equivalences
\begin{eqnarray*}
\Sigma(\hocolim{\nN(q,G)}BA)&\cong& \hocolim{\nN(q,G)}\Sigma(BA)\\
&\simeq & \hocolim{\nN(q,G)}\bigvee_{p||A|}BA_{(p)} \\
&\simeq & \bigvee_{p||G|}\hocolim{\nN(q,G)}B\pi_p(A)\\
&\simeq & \bigvee_{p||G|}\hocolim{\nN(q,G)_p}BP.
\end{eqnarray*}
When commuting the wedge product with the homotopy colimit, again an argument using the commutativity of homotopy colimits can be used similar to the suspension case. Note that the wedge product is a homotopy colimit. 
}
This stable equivalence immediately implies the following decomposition for a generalized  cohomology theory.
\begin{thm}\label{p-decomposition}
There is an isomorphism
\[
\tilde{h}^*(B(q,G))\cong \prod_{p||G|} \tilde{h}^*(B(q,G)_p)\; \text{ for all }q\geq 2
\]
where $\tilde{h}^*$ denotes a reduced cohomology theory. 
\end{thm}
Therefore one can study homological properties of  $B(q,G)$ at a fixed prime. In particular, this theorem applies  to the complex $K$--theory of $B(q,G)$ and each piece corresponding to $B(q,G)_p$ can be computed using the Bousfield--Kan spectral sequence \cite{B-K}. 

\subsection{Homotopy types of $B(q,G)$ and $B(q,G)_p$}
We follow the discussion on the fundamental group of $B(q,G)$ from \cite{adem1}, similar properties are satisfied by $B(q,G)_p$. Let $\nN$ be a collection of subgroups of a finite group $G$. We require that $\nN$ has an initial element and is closed under taking subgroups. In particular, it can be taken as $\nN(q,G)$ or $\nN(q,G)_p$.
The fundamental group of a homotopy colimit is described in \cite[Corollary 5.1]{far}. There is an isomorphism
\[
\pi_1(\hocolim{\nN} BA)\cong \colim{\nN} A
\]
where the colimit is in the category of groups.
 We will study colimits of (abelian) groups  in the next section in more detail.

 For each $A\in \nN$, there is a  natural fibration $G/A\rightarrow BA\rightarrow BG$ and taking homotopy colimits one obtains a fibration
\begin{eqnarray}
\hocolim{\nN}G/A\rightarrow \hocolim{\nN}BA\rightarrow BG\label{eq:fibration}.
\end{eqnarray}
This is a consequence of a Theorem of  Puppe,  see \cite[Appendix HL]{far2}. The point is that each fibration has the same base space $BG$. Associated to this fibration  there is an exact sequence  of homotopy groups
\begin{eqnarray}\label{eq:homotopyex}
1\rightarrow \pi_1(\hocolim{\nN}G/A)\rightarrow \pi_1(\hocolim{\nN}BA)\stackrel{\psi}{\rightarrow} \pi_1(BG)\rightarrow \pi_0(\hocolim{\nN} G/A) \rightarrow 0,
\end{eqnarray}
where $\psi $ is the natural map 
\[
\pi_1(\hocolim{\nN}BA)\cong \colim{\nN} A\rightarrow G
\]
induced by the inclusions $A\rightarrow G$. 
Note that if $\nN$ contains all cyclic subgroups of $G$ then  $\hocolim{\nN}G/A$ is connected. Since in this case, the commutativity of 
\[
\begin{diagram}
\colim{\nN}A &\rTo^\psi& G\\
\uTo & \ruTo &\\
\Span{g}
\end{diagram}
\] 
implies that $\psi$ is surjective.

We point out here that for $\nN=\nN(q,G)$ the homotopy fibre $\hocolim{\nN}G/A$ is  homotopy equivalent to the pull-back of the universal principal $G$--bundle 
\[
\begin{CD}
E(q,G) & @>>> & EG \\
@VVV & & @VVV \\
B(q,G) & @>>>& BG
\end{CD}
\]
and this pull-back $E(q,G)$ can be defined as the geometric realization of a simplicial set as described in \cite{adem1}. It can be identified as a colimit
\[
E(q,G)\cong \colim{\nN(q,G)}G\times_A EA
\] 
and there is a commutative diagram 
\[
\begin{CD}
\hocolim{\nN(q,G)} G\times_A EA &@>\sim >>& \hocolim{\nN(q,G)}G/A\\
@VV\sim V &&@VVV\\
\colim{\nN(q,G)} G\times_A EA  & @>>> & \colim{\nN(q,G)}G/A
\end{CD}
\]
induced by the contractions $EA\rightarrow \pt$, and the weak equivalences as indicated. 
The exact sequence (\ref{eq:homotopyex}) becomes
\begin{eqnarray}\label{eq:homotopy}
0\rightarrow T(q)\rightarrow G(q)\stackrel{\psi}{\rightarrow} G\rightarrow 0,
\end{eqnarray}
where $G(q)=\pi_1(B(q,G))$ and $T(q)=\pi_1(E(q,G))$. 

Consider  the universal cover $\widetilde{B(q,G)}$ of $B(q,G)$. Again using the Theorem of Puppe, it can be described as the homotopy fibre of the homotopy colimit of the natural fibrations 
$ G(q)/A\rightarrow BA\rightarrow BG(q) $ (\cite[Theorem 4.4]{adem1})
\[
\begin{CD}
\widetilde{B(q,G)}\simeq\hocolim{\nN(q,G)}G(q)/A &@>>>&B(q,G)\\
 &@. & @VVV\\
 &@. & BG(q).
\end{CD}
\]
The question of $B(q,G)$ having the homotopy type of a $K(\pi,1)$ space is equivalent to asking whether  the classifying space functor $B$ commutes with colimits
\[
B(q,G)=\colim{\nN(q,G)}BA\rightarrow B(\colim{\nN(q,G)}A).
\]
The difference is measured by the simply connected, finite dimensional complex $\widetilde{B(q,G)}$, or equivalently by the values of the higher limits 
\[
H^i(\widetilde{B(q,G)};\ZZ)\cong \ilim{\nN(q,G)}{i} \ZZ[G(q)/A]
\]
as implied by the Bousfield--Kan spectral sequence. 

The class of extraspecial $2$--groups provides examples for which $B(2,G)$ is not  homotopy equivalent to a $K(\pi,1)$ space. 
This follows from  the computation of the fundamental group of $B(2,G)$ that is the colimit of abelian subgroups of $G$, see \S \ref{sec:colim extra} and \S \ref{sec:examples}.

\section{Higher limits}\label{sec:higherlimits}
As the higher limits appear in the $E_2$--term of the Bousfield--Kan spectral sequence, in this section we discuss some of their properties. The main theorem of this section is a vanishing result of the higher limits of the contravariant part of a pre-Mackey functor $R\co\catdn\rightarrow \catAb$ where $\catdn$ denotes the poset of non-empty subsets of $\set{1,...,n}$ ordered by reverse inclusion. 

Let $\mathbf{C}$ be a small category and $F\co\catC \rightarrow \catAb$ be a \textit{contravariant} functor from $\catC$ to the category of abelian groups. $\catAb^\catC$ denotes the category of contravariant functors $\catC\rightarrow \catAb$.
Observe that
\[
\ilim{}{} F\cong \Hom_\catAb(\ZZ ,\ilim{}{} F)\cong \Hom_{\catAb^\catC}(\underline{\ZZ} ,F).
\]
\Def{\rm{ Derived functors of the inverse limit of $F\co\catC\rightarrow \catAb$ are defined by
\[
\ilim{}{i} F\equiv \Ext^i_{\catAb^\mathbf{C}}(\underline{\ZZ},F). 
\]
}}
A projective resolution of $\underline{\ZZ}$ in $\catAb^\catC$  can be obtained in the following way, for details see \cite{webb}. First note that the functors $F_c\co\catC\rightarrow \catAb$ defined by $F_c(c')=\ZZ \; \Hom_\mathbf{C} (c',c)$ are projective functors and by Yoneda's lemma 
\[
\Hom_{\catAb^\catC}(F_c,F)\cong F(c).
\] 
Let $\catC\setminus -\co\catC\rightarrow \catS$ denote the functor which sends an object $c$ of $\catC$ to the nerve of the under category $\catC\setminus c$. The nerve $B(\catC\setminus c)$ is contractible since the object $c\rightarrow c$ is initial. Then we define a resolution $P_*\rightarrow \underline{\ZZ}$ of the constant functor as the composition $ P_*=C_*\circ \catC\setminus -$ where $C_*\co\catS\rightarrow \catAb$ is the functor which sends a simplicial set $X$ to the associated chain complex  $C_*(X)=\ZZ [X_*]$.
At each degree $n$, there are  isomorphisms
\[
\Hom_{\catAb^\catC}(P_n,F)\cong \Hom_{\catAb^\catC}(\bigoplus_{x_0\leftarrow \cdots \leftarrow x_n} F_{x_n},F)\cong \underset{x_0\leftarrow x_1\leftarrow \cdots\leftarrow x_n}{\prod}F(x_n).
\]
Differentials are induced by the simplicial maps between chains of morphisms. The following result describes them explicitly.
\begin{lem}[\cite{oliver}]\label{lim-complex}
 $\ilim{}{i}F\cong H^i(C^*(\mathbf{C};F),\delta ) $ where
$$
C^n(\mathbf{C};F)=\underset{x_0\leftarrow x_1\leftarrow \cdots\leftarrow x_n}{\prod}F(x_n)
$$
for all $n\geq 0$ and where for $U \in C^n(\mathbf{C};F)$

\begin{eqnarray*}
\delta (U)(x_0\leftarrow x_1\leftarrow \cdots\leftarrow x_n\stackrel{\phi}{\leftarrow}x_{n+1})=\sum_{i=0}^n(-1)^iU(x_0\leftarrow \cdots\leftarrow \hat{x}_i\leftarrow \cdots\leftarrow x_{n+1}) \\
+(-1)^{n+1}F(\phi )(U(x_0\leftarrow \cdots\leftarrow x_n)).
\end{eqnarray*}
\end{lem}

A useful property of higher limits, which allows the change of the indexing category, is the following.
\begin{pro}\cite[Lemma 3.1]{JM}\label{properties}
Fix a small category $\mathbf{C}$ and a contravariant functor $F\co\mathbf{C}\rightarrow \mathbf{Ab}$. Let $\mathbf{D}$ be a small category and assume that $g\co\mathbf{D}\rightarrow \mathbf{C}$ has a left adjoint. Set $g^*F=F\circ g \co \mathbf{D}\rightarrow \mathbf{Ab}$ then
$
H^*(\mathbf{C};F)\cong H^*(\mathbf{D};g^*F).
$

\end{pro}

Assume that the indexing category $\mathbf{C}$ is a finite partially ordered set (poset) whose morphisms are $i\rightarrow j$ whenever $i\leq j$. We consider an appropriate filtration of $F$ as in \cite{grodal}. 
\Def{\rm{
A \textit{height function} on a poset is a strictly increasing map of posets $ht\co\mathbf{C}\rightarrow \ZZ$. 
}}
A convenient height function can be defined by $ht(A)= - \Dim (| \mathbf{C}_{\geq A} |)$ where $\mathbf{C}_{\geq A}$ is the subposet of $\mathbf{C}$ consisting of elements $C\geq A$. Let  $N=\Dim (|\mathbf{C}|)$ then it is immediate that $\ilim{}{i} F$ vanishes for $i>N$. The functor $F$ can be filtered in such a way that the associated quotient functors are concentrated at  a single height. Define a sequence of subfunctors
$$
 F_{N}\subset \cdots\subset F_{2} \subset F_{1}\subset F_0
$$ 
by  $F_{i}(A)=0$ if $ht(A)>-i$ and $F_{i}(A)=F(A) $ otherwise. This induces a decreasing filtration on the cochain complexes hence there is an associated spectral sequence whose $E_0$--term is
$$
E^{i,j}_0=C^{i+j}(\mathbf{C};F_i/F_{i+1})= \prod_{A\in \mathbf{C}|ht(A)=-i} \Hom _{\ZZ} ( C_{i+j}(|\mathbf{C}_{\geq A}|,|\mathbf{C}_{> A} |),F(A) ) 
$$
with differentials $d_0\co C^k(\mathbf{C};F_i/F_{i+1})\rightarrow C^{k+1}(\mathbf{C};F_i/F_{i+1})$. $E_1$--term is the  cohomology of the pair $(|\mathbf{C}_{\geq A}|,|\mathbf{C}_{> A} |)$ in coefficients $F(A)$ that is
\begin{eqnarray}\label{spec}
E^{i,j}_1= 
\prod_{A\in \mathbf{C}|ht(A)=-i} H^{i+j} ((|\mathbf{C}_{\geq A}|,|\mathbf{C}_{> A} |);F(A) ).
\end{eqnarray}
Fix $A$ and let $A'>A$ such that $ht(A')=ht(A)+1$, then differential $d_1$ can be described as the composition
\begin{eqnarray*}
H^k((|\mathbf{C}_{\geq A'}|,|\mathbf{C}_{>A'}|);F(A'))\stackrel{F(A\leq A')}{\rightarrow}
H^k((|\mathbf{C}_{>A}|,|\mathbf{C}_{>A,ht\geq ht(A)+2}|);F(A))\\
\stackrel{\partial}{\rightarrow}
H^{k+1}((|\mathbf{C}_{\geq A}|,|\mathbf{C}_{>A}|);F(A))
\end{eqnarray*}
where $\partial$ is the boundary map associated to the triple 
$$(|\mathbf{C}_{\geq A}|,|\mathbf{C}_{>A}|,|\mathbf{C}_{>A,ht\geq ht(A)+2}|). $$

We now consider higher limits over specific kinds of diagrams. Let $\mathbf{dn}$ be the category associated to the non-degenerate simplexes of the standard $n$--simplex, the objects are increasing sequences of numbers  $\sigma _k= [i_0<\cdots<i_k]$ where $0\leq i_j\leq n$ and the morphisms are generated by the face maps $d_j([i_0<\cdots<i_k])=[i_0<\cdots<\hat{i_j}<\cdots<i_k]$ for $0\leq j\leq k$. 

\begin{pro}\label{cech}
For any contravariant functor $F\co\mathbf{dn}\rightarrow \mathbf{Ab}$ the spectral sequence (\ref{spec}) collapses onto the horizontal axis hence gives a long exact sequence
\begin{eqnarray}\label{complex}
0\rightarrow \prod_{\sigma _0} F(\sigma _0)\rightarrow \prod_{\sigma _1} F(\sigma _1) \rightarrow \cdots
 \rightarrow \prod_{\sigma _{n-1}} F(\sigma _{n-1}) \rightarrow F(\sigma _n)\rightarrow 0
\end{eqnarray}
 and where for $U \in C^{k-1}_{\mathbf{dn}}(F)=\prod_{\sigma_{k-1}} F(\sigma _{k-1})$,

\begin{eqnarray*}
\delta ^{k-1}(U)(\sigma _k)=\sum_{j=0}^{k}(-1)^{k-j} F(d_j)U(d_j(\sigma _k)) .
\end{eqnarray*}

\end{pro}
\begin{proof}
For a simplex $\sigma _k \in \mathbf{dn}$, let $\mathbf{dn}_{\geq k}$ be the poset of simplices $\sigma \geq \sigma _k$ and $\mathbf{dn}_{>k}$ denote the poset of simplices $\sigma > \sigma _k$. Observe that the pair $(|\mathbf{dn}_{\geq k}|,|\mathbf{dn}_{> k}|)$ is homeomorphic to $(|\Delta ^k|,|\partial \Delta ^k|)$ where $\Delta ^k$ is the standard $k$--simplex with boundary $\partial \Delta ^k$. 
The spectral sequence of the filtration has
$$
E^{k,j}_1= \prod_{\sigma _k} H^{k+j} ((|\mathbf{dn}_{\geq k}|,|\mathbf{dn}_{> k}|);F(\sigma _k) ).
$$
Therefore $E_1$--term vanishes unless $j=0$, and otherwise 
\[ E^{k,0}_1= \prod_{\sigma _k} H^k ((|\mathbf{dn}_{\geq k}|,|\mathbf{dn}_{> k}|);F(\sigma _k) )=\prod_{\sigma _k} F(\sigma _k) \]
hence collapses onto the horizontal axes, resulting in a long exact sequence.
The differential is induced by the alternating sum of the face maps $d_j(\sigma_k)$ for $0\geq j\geq k$.
\end{proof}

In some cases, it is possible to show that the higher limits vanish. The next theorem is an illustration of this instance. First let us recall the definition of a \textit{pre-Mackey functor} in the sense of Dress \cite{dress}. Let $M\co\catC\rightarrow \catD$ be a bifunctor, that is a pair of functors $(M^*,M_*)$ such that $M^*$ is contravariant, $M_*$ is covariant and both coincide on objects.

\Def{\rm{ A \textit{pre-Mackey functor} is a bifunctor $M\co\catC\rightarrow\catD$ such that for any pull-back diagram 
 \[
\begin{CD}
X_1 @>\beta_1>> X_2   \\
@V\alpha_1VV @V\beta_2VV       \\
X_3 @>\alpha_2>> X_4  
\end{CD}
\]
in $\catC$  the diagram
\[
\begin{CD}
M(X_1) @>M_*(\beta_1)>> M(X_2) \\
@AM^*(\alpha_1)AA @AM^*(\beta_2)AA  \\
M(X_3) @>M_*(\alpha_2)>> M(X_4)
\end{CD}
\]
commutes.
}}
\noindent We give a direct proof of the following theorem, which can also be deduced from an application of \cite[Theorem 5.15]{JM} to the category $\mathbf{dn}$.
\begin{thm}\label{torsion}
Let $R\co\mathbf{dn}\rightarrow \mathbf{Ab}$ be a pre-Mackey functor then
 $\ilim{}{s} R^*=0$ for $s>0$.
\end{thm}
\begin{proof}
We do induction on $n$. The case $n=0$ is trivial since the complex (\ref{complex}) is concentrated at degree $0$. There are inclusions of categories 
\[ \iota_0\co\mathbf{dn-1}\rightarrow \mathbf{dn}\]
 defined by $[i_1<...<i_k]\mapsto [i_1+1<...<i_k+1]$ and 
 \[ \iota_1\co\mathbf{dn-1}\rightarrow \mathbf{dn} \] 
 where  $[i_1<...<i_k]\mapsto [0<i_1+1<...<i_k+1]$.
Let $C^*_{\mathbf{dn}}$ denote the complex $C^*_{\mathbf{dn}}(R)$ in (\ref{complex}) that is
$
C^k_{\mathbf{dn}}=\prod_{\sigma _k} R(\sigma _k).
$ Define a filtration $C^*_{-1}\subset C^*_0\subset C^*_\mathbf{dn}$ such that 
$$C^k_{0}=\prod_{\sigma _k=[0<i_1<...<i_k]} R(\sigma _k)$$
 and $C^k_{-1}=C^k_0$ for $k>0$ and $C^0_{-1}=0$. Then by induction $H^i(C^*_\mathbf{dn}/C^*_0)=H^i(C^*_\mathbf{dn-1}(R\circ\iota_0))=0$ for $i>0$ and $H^i(C^*_{-1})=H^{i-1}(C^*_\mathbf{dn-1}(R\circ \iota_1))=0$ for $i>1$.
And short exact sequences associated to the filtration $C^*_{-1}\subset C^*_0\subset C^*_\mathbf{dn}$ implies that it suffices to prove $H^1(C^*_{0})=0$.

For a simplex $\sigma =[0<i<...<j]$ and morphisms $\alpha_i \co\sigma \rightarrow [0<i]$ and $\alpha_j \co\sigma \rightarrow [0<j]$ in $\mathbf{dn}$,  set $\phi _{\sigma }= R_*(\alpha_j )R^*(\alpha_i ) $. 

For a fixed $i$, define a map $\Phi _i\co R([0<i])\rightarrow R([0])$ by
\begin{eqnarray*}
\Phi _i=  \sum_{\sigma =[0<i<...<j]} (-1)^{|\sigma |-1} R_*(d^j_1)\phi _{\sigma} \;\;\text{ where }\;\;
d^j_1\co[0<j]\rightarrow [0],
\end{eqnarray*}
where $|\sigma|$ denotes the dimension of the simplex $\sigma$. Combining these maps for all $1\leq i \leq n$  define $\Phi \co \prod_{l=1}^n R([0<l])\rightarrow R([0])$ by
\begin{eqnarray*}
\Phi =\sum_{i=1}^{n} \Phi _i\pi _i,
\end{eqnarray*}
where $\pi _i$ denotes the natural projection $\prod_{l=1}^nR([0<l])\rightarrow R([0<i])$ to the $i^{th}$ factor. 

We need to show that if $U$ is in the kernel of $\delta ^1$, then $\delta ^0 \Phi (U)=U$ or equivalently $R^*(d^i_1)\Phi (U)= \pi _i (U)$ for all $1\leq i \leq n$. 
Fix $k$, and compute the composition 
\begin{eqnarray*}
R^*(d^k_1)\Phi &=& \sum_{i=1}^{n} \sum_{\sigma =[0<i<...<j]}(-1)^{|\sigma |-1} R^*(d^k_1)R_*(d^j_1)\phi _{\sigma}\\
&=&\pi _k + \sum_{i=1}^{n} \left( \sum_{\sigma |k\in \sigma \neq [0<k]}(-1)^{|\sigma |-1}R^*(d^k_1)R_*(d^j_1)\phi _{\sigma}+ \sum_{\sigma |k\notin \sigma } (-1)^{|\sigma |-1}R^*(d^k_1)R_*(d^j_1)\phi _{\sigma} \right)
\end{eqnarray*}
in the second line first use the identity $ R^*(d^k_1)R_*(d^j_1)\phi _{[0<k]}=\pi_k $ then separate the summation into two parts such that it runs over the simplices which  contain $\{ k \}$ and do not contain $\{ k\}$. Suppose that $\sigma =[0<i<...<j]$ does not contain $\{ k \}$ and $\overline{\sigma}$ denote the simplex obtained from $\sigma$ by adjoining $\{ k \}$. There are three possibilities for $\overline{\sigma}$: 
\[
\overline{\sigma}= \left\{
\begin{array}{l l}
\text{$[0<i<...<k<...<j]$} & \text{ if } i<k<j, \\
 \text{$[0<i<...<j<k]$} & \text{ if } j<k, \\
\text{$[0<k<i<...<j]$} & \text{ if } k<i.
\end{array}\right.
\]
The corresponding terms of the first two cases cancel out when $\sigma$ and $\overline{\sigma}$ are paired off in the summation, since the following diagram commutes 
\[
\begin{CD}
 R(\overline{\sigma}) @<<< R([0<i]) @>R^*(\alpha_i )>> R(\sigma )\\
@VVV @. @VR_*(\alpha_j )VV\\
R([0<k]) @<R^*(d^k_1) << R([0]) @<R_*(d^j_1)<< R([0<j])
\end{CD}
\]
and the corresponding terms become equal with opposite signs. So the equation simplifies to
\begin{eqnarray*}
R^*(d^k_1)\Phi &=& \pi _k+ \sum_{\sigma =[0<k<i<...<j]}(-1)^{|\sigma |-1}( R^*(d^k_1)R_*(d^j_1)\phi _{\sigma}- R^*(d^k_1)R_*(d^j_1)\phi _{d_1\sigma})
\end{eqnarray*}
where the sum runs over the simplices starting with $\{ 0<k\}$.  The 
 maps $\sigma \stackrel{\theta }{\rightarrow} [0<k<i]$, $ \sigma \stackrel{\vartheta }{\rightarrow} [0<k]$ and also the two face maps $[0<i] \stackrel{d_1}{\leftarrow} [0<k<i]\stackrel{d_2}{\rightarrow} [0<k] $ in our indexing category $\mathbf{dn}$ give a commutative diagram
 \[
\begin{CD}
 R([0<i]) @>R^*(d_1)>> R([0<k<i]) @>R^*(\theta )>> R(\sigma )\\
@VVV @. @VR_*(\vartheta )VV\\
R(d_1\sigma ) @>>> R([0<j]) @>R^*(d^k_1)R_*(d^j_1)>> R([0<k]).
\end{CD}
\]
Hence we can write
\begin{eqnarray*}
R^*(d^k_1)\Phi &=& \pi _k+ \sum_{\sigma =[0<k<i<...<j]} (-1)^{|\sigma |-1}R_*(\vartheta )R^*(\theta)\left( R^*(d_2)-R^*(d_1) \right) .
\end{eqnarray*}
Now suppose that $U$ is in the kernel of $\delta ^1\co C^1_{0}\rightarrow C^2_{0}$,  the summation above vanishes since $\left( R^*(d_2)-R^*(d_1) \right)(U)=0$. Therefore $R^*(d^k_1)\Phi (U)=\pi _k(U)$, which implies that $H^1(C^*_{0})=0$. 
\end{proof}

\subsection{Higher limits of the representation ring functor}

Let $G$ be a finite group and $p$ be a prime dividing its order. Let $M_0,...,M_n$ denote the maximal abelian $p$--subgroups of $G$. Consider the contravariant functor $R\co\mathbf{dn}\rightarrow \catAb$ defined by $\sigma \mapsto R(M_\sigma )$ where $M_\sigma=\cap_{i\in \sigma}M_i$ and $R(H)$ is the Grothendieck ring of complex representations of $H$. Note that $\QQ\otimes R$ is a pre-Mackey functor: If $\alpha\co \sigma\rightarrow \sigma'$ set $A=\cap_{i\in \sigma} M_i$ and $B=\cap_{i\in\sigma'}M_i$ then $(\QQ\otimes R^*(\alpha),\QQ\otimes R_*(\alpha))=(\res_{B,A},\frac{1}{|B:A|}\indu_{A,B})$. Hence Theorem \ref{torsion} implies the following.  

\Cor{\label{vanish} If $R\co\mathbf{dn}\rightarrow \catAb$ defined as above then $\ilim{}{s} \QQ\otimes R^*=0$ for $s>0.$}
\Proof{We give an alternative proof which is specific to the representation ring functor. Let $U\co
\catGrp\rightarrow \catSet$ denote the forgetful functor which sends a group to the underlying set. The isomorphism $\CC\otimes R(M_\sigma)\cong \Hom_\catSet(M_\sigma,\CC)=H^0(U(M_\sigma);\CC))$ induces an isomorphism $\ilim{}{s}\CC\otimes R(M_\sigma)\cong \ilim{}{s} H^0(U(M_\sigma);\CC)$. Note that $U(M_\sigma)$ is a finite set of points. Consider the space $\hocolim{\catdn} U(M_\sigma)$ and the associated cohomology spectral sequence 
\[ E_2^{s,t}=\ilim{\catdn}{s}H^t(U(M_\sigma);\CC)\Rightarrow H^{s+t}(\hocolim{\catdn} U(M_\sigma);\CC).\]
Now, the diagram $\catdn\rightarrow \catTop$ defined by $\sigma\mapsto U(M_\sigma)$ is free (\S \ref{sec:app}), consists of inclusions of finite sets. Therefore the natural map $\hocolim{\catdn} U(M_\sigma)\rightarrow \colim{\catdn} U(M_\sigma)=U(G)$ is a homotopy equivalence, note that the latter is a finite set. The spectral sequence collapses at $E_2$--page and gives $\ilim{}{s}H^0(U(M_\sigma);\CC)\cong H^{s}(\hocolim{\catdn} U(M_\sigma);\CC)$ which vanishes for $s>0$. 
}

\section{$K$--theory of $B(2,G)$}\label{sec:ktheory}
In this section, we assume that $G$ is a finite group. 
We study the complex $K$--theory of $B(q,G)$ when $q=2$. Theorem \ref{p-decomposition} gives a decomposition
\[
\tilde{K}^*(B(2,G))\cong \bigoplus_{p||G|} \tilde{K}^*(B(2,G)_p),
\]
of the reduced $K$--theory of $B(2,G)$.
For each such prime $p$ recall the fibration $\hocolim{\nN(q,G)_p} G/P\rightarrow B(q,G)_p \rightarrow BG$ from \S \ref{sec:finite}.
Furthermore, the Borel construction commutes with homotopy colimits and gives a weak equivalence 
\[
B(q,G)_p\simeq (\hocolim{\nN(q,G)_p}G/P)\times_G EG. 
\]
There is the Atiyah--Segal completion theorem \cite{a-s} which relates the equivariant $K$--theory of a $G$--space $X$ to the complex $K$--theory of its Borel construction $X\times_G EG$.

\subsubsection*{Equivariant $K$-theory} Complex equivariant $K$--theory is a $\ZZ /2$--graded cohomology  theory defined for compact spaces using $G$--equivariant complex vector bundles \cite{segal}. Recall that
\[
K^n_G(G/H)\cong \left\lbrace
\begin{array}{cc}
R(H) & \text{ if  } n=0,\\
0 & \text{ if } n=1, 
\end{array} \right.
\]
where $H\subseteq G$ and recall that $R(H)$ denotes the Grothendieck ring of complex representations of $H$. Note that in particular, $K^0_G(\pt )\cong R(G)$ and $K^*_G(X)$ is an $R(G)$--module via the natural map $X\rightarrow \pt$. 

Let $X$ be a compact $G$--space and $I(G)$ denote the kernel of the augmentation map $R(G)\stackrel{\varepsilon}{\rightarrow} \ZZ$.
The Atiyah--Segal completion theorem \cite{a-s} states that $K^*( X\times_G EG )$ is the completion of $K^*_G(X)$ at the augmentation ideal $I(G)$. In particular, taking $X$ to be a point this theorem implies that 
\[
 K^n(BG) \cong
 \left\{ 
 \begin{array}{cc}
R(G)^\wedge & n=0, \\
0& n=1. 
\end{array} 
\right.
\]
The completion $R(G)^\wedge$ can be described using the restriction maps to a Sylow $p$--subgroup for each $p$ dividing the order of the group $G$. Let $I_p(G)$ denote the quotient of $I(G)$ by the kernel of the restriction map $I(G)\rightarrow I(G_{(p)})$ to a Sylow $p$--subgroup. There are isomorphisms of abelian groups
\[
 K^n(BG) \cong
 \left\{ 
 \begin{array}{cc}
\ZZ\oplus \bigoplus_{p|\;|G|} \ZZ_p\otimes I_p(G) & n=0, \\
0& n=1,
\end{array} 
\right.
\]
see \cite{luck}. Note that if $G$ is a nilpotent group, equivalently $G\cong \prod_{p|\;|G|}G_{(p)}$, the restriction map $I(G)\rightarrow I(G_{(p)})$ is surjective and
\begin{eqnarray}\label{nilcomp}
\tilde{K}^0(BG)\cong I(G)^\wedge\cong \bigoplus_{p|\;|G|} \ZZ_p\otimes I(G_{(p)}).
\end{eqnarray}
We also use the fact that the completion of $R(H)$ with respect to $I(H)$ is isomorphic to the completion with respect to $I(G)$ as an $R(G)$--module via the restriction map $R(G)\rightarrow R(H)$.

\subsection{$K$--theory of $B(2,G)$}  The $E_2$--page of the equivariant version of the Bousfield--Kan spectral sequence \cite{lee} for the equivariant $K$--theory of  $\hocolim{\nN(q,G)_p} G/P $ is given by
\begin{equation}\label{E2}
E^{s,t}_2\cong \ilim{\nN(q,G)_p}{s}K^t_G(G/P) 
\cong \left\{
\begin{array}{ll}
\ilim{\nN(q,G)_p}{s}R(P)&  \text{ if $t$ is even, }  \\
0 &\text{ if $t$ is odd, } 
\end{array} \right.
\end{equation}
and the spectral sequence converges to $K^*_G(\hocolim{\nN(q,G)_p} G/P )$ whose completion is $K^*(\hocolim{\nN(q,G)_p} BP)$. We will split off the trivial part of $E_2^{s,t}$  and consider 
$$
\tilde{E}_2^{s,t}\cong \ilim{\nN(q,G)_p}{s}I(P)\;\; \text{ for $t$ even.}
$$
Note that the splitting $R(P)\cong \ZZ\oplus I(P)$ gives $\ilim{\nN(q,G)_p}{s}R(P)\cong \ilim{\nN(q,G)_p}{s}I(P)$ for $s>0$ since the nerve of $\nN(q,G)_p$ is contractible. The following result describes the terms of this spectral sequence for the case $q=2$.

\begin{lem}\label{E2tor}
Let  $R\co\nN(2,G)_p\rightarrow \mathbf{Ab}$ be the representation ring functor defined by $P\mapsto R(P)$ then  $\ilim{}{s} R$ are torsion groups for $s>0$. Furthermore, $\ilim{}{0}R\cong \ZZ^{n_p+1}$ where $n_p$ is the number of (non-identity) elements of order a power of $p$ in $G$.
\end{lem}
\Proof{
Let $M_i$ for $0\leq i \leq n$ denote the maximal abelian $p$--subgroups of $G$. Define a functor  $g\co\mathbf{dn}\rightarrow \nN(2,G)_p $ by sending a simplex $[i_0<i_1<...<i_k]$ to the intersection  $\bigcap_{j=0}^{k} M_{i_j} $. Then $g $ has a left adjoint, namely the functor $P \mapsto \bigcap_{M_i : P\subset M_i } M_i$ sending an abelian $p$--subgroup to the intersection of maximal abelian $p$--subgroups containing it. By  Proposition \ref{properties} we can replace the indexing category $\nN(2,G)_p$ by $\mathbf{dn}$ and calculate the higher limits over the category $\mathbf{dn}$. Then the first part of the theorem follows from Corollary \ref{vanish}. 

For the second part note that $\ilim{}{0} R=\ilim{}{} R $   is a submodule of a free $\ZZ$--module, hence it is free. Then it suffices to calculate its rank. Tensoring with $\CC$ gives an isomorphism 
\[\CC\otimes \ilim{\nN(2,G)_p}{0}R(P)\cong \Hom_\catSet(\cup_{\nN(2,G)_p} P,\CC).\]
 Recall that $\nN(2,G)_p$ is the collection of abelian $p$--subgroups of $G$. The result follows from the identity $|\cup_{\nN(2,G)_p} P|=n_p+1$ where $n_p$ is the number of (non-identity) elements of order a power of $p$ in $G$.
}

\begin{thm}\label{K(B(2,G))}
 There is an isomorphism
\[
\QQ \otimes K^i(B(2,G))\cong
\left\lbrace
\begin{array}{ll}
\QQ \oplus \bigoplus_{p\mid |G|} \QQ_p^{n_p} & \text{ if $i=0$, }\\
0 & \text{ if $i=1$, }
\end{array} \right.
\]
where $n_p$ is the number of (non-identity) elements of order a power of $p$ in $G$.
\end{thm}
\Proof{  
By the decomposition in Theorem \ref{p-decomposition} we  work at a fixed prime $p$ dividing the order of $G$. Then $\tilde{K}^*(B(2,G)_p)$ is isomorphic to the completion of $ \tilde{K}^*_G(\hocolim{\nN(2,G)_p} G/P)$ at the augmentation ideal $I(G)$ by the Atiyah--Segal completion theorem.  Since $R(G)$ is a Noetherian ring the completion $-\otimes_{R(G)}R(G)^\wedge$ is exact on finitely generated $R(G)$--modules, hence commutes with taking homology. Moreover,  (\ref{nilcomp}) gives
$$
R(P)\otimes_{R(G)}R(G)^\wedge \cong R(P)^\wedge \cong \ZZ\oplus I(P)\otimes\ZZ_p .
$$
This isomorphism induces 
$$
\tilde{E}_r^{s,t}\otimes_{I(G)}I(G)^\wedge \cong \tilde{E}_r^{s,t}\otimes_{\ZZ} \ZZ_p\;\; \text{ for $r\geq 2$,}
$$
which gives an isomorphism between the abutments. For $t$ even and $s>0$, $\tilde{E}_2^{s,t}$ are torsion by Lemma \ref{E2tor}. After tensoring with $\QQ$, the spectral sequence  collapses onto the vertical axis at the $E_2$--page and $\QQ\otimes\tilde{K}^0(B(2,G)_p)\cong  \QQ\otimes \ZZ_p^{n_p}$ whereas $\QQ \otimes K^1(B(2,G)_p)$ vanishes. 
}

\section{Colimits of abelian groups}\label{sec:colimits}

The fundamental groups of $B(2,G)$ and $B(2,G)_p$ are isomorphic to the colimits of the  collection of abelian subgroups  and abelian $p$--subgroups of $G$, respectively. Therefore it is natural to study the colimit of a collection of abelian groups in greater generality to determine the homotopy properties of these spaces.

Let $\aA$ denote  a collection of abelian groups closed under taking subgroups. We  consider the colimit of the groups in $\aA$. This can be constructed as a quotient of the free product $\coprod$  of the groups in the collection
\[
\colim{\aA} A\cong  (\coprod_{A\in \aA} A) /\sim
\]
by the normal subgroup generated by the relations $b\sim f(b)$ where $b\in B$ and $f\co B\rightarrow A$ runs over the morphisms in $\aA$. Let $T$ be a finite abelian group, define the number 
\[
d(T)=\sum_{p|\; |T|} \rank(T_{(p)})
\]
where $T_{(p)}$ denotes the Sylow $p$--subgroup of $T$.  
For $r>0$ and a collection of finite abelian groups $\aA$, define a sub-collection 
\[
\aA_r=\lbrace A\in \aA|\; d(A)\leq r \rbrace.
\] 
The main result of this section is the following isomorphism of colimits, which reduces the collection $\aA$ to the sub-collection $\aA_2$. Note that if $\aA$  consists of elementary abelian $p$--groups then $\aA_2$ is the sub-collection of groups of rank at most $2$.   

\begin{thm}\label{rank2}
The natural map
\[
\colim{\aA_2} A \rightarrow \colim{\aA} A
\]
induced by the inclusion  map $\aA_2\rightarrow \aA$ is an isomorphism.
\end{thm}
 
The proof follows from the homotopy properties of a complex constructed from the cosets of proper subgroups of a group. We review some background first.

\subsubsection*{The coset poset} Consider the poset $\lbrace xH|\; H\subsetneq G \rbrace$ consisting of cosets of proper subgroups of $G$ ordered by inclusion. The associated complex is denoted by $\cC(G)$ and it is studied in \cite{brown}. This complex can be identified as a homotopy colimit
\[
\cC(G)\simeq \hocolim{\set{H\subsetneq G}}G/H.
\]
The identification follows from the description of this homotopy colimit as the nerve of the transport category of the poset $\set{H\subsetneq G}$ which is precisely the nerve of the poset $\lbrace xH|\; H\subsetneq G \rbrace$.

 When $G$ is solvable, $\cC(G)$ has the  homotopy type of a bouquet of spheres. To make this statement precise we need a definition from group theory.
 
 \Def{\rm{ A \textit{chief} series of a group $G$ is a series of  normal subgroups
\[
1=N_0\subset N_1\subset ...\subset N_k=G
\]
for which each factor $N_{i+1}/N_i$ is a  minimal (nontrivial) normal subgroup of $G/N_i$. 
}}
\begin{pro}\cite[Proposition 11]{brown}\label{cosetposet}
Suppose that $G$ is a solvable finite group and 
\[
1=N_0\subset N_1\subset ...\subset N_k=G
\]
be a chief series then
\[
\mathcal{C}(G)\simeq \bigvee^n S^{d-1}
\]
for some $n>0$, where $d$ is the number of indices $i=1,2,...,k$ such that $N_{i}/N_{i-1}$ has a complement in $G/N_{i-1}$.
\end{pro}
We are  interested only in the dimensions of the spheres. 

\subsubsection*{Coset poset of abelian groups} The number $d$ in Proposition \ref{cosetposet} for an abelian group $T$  is given by $d(T)$ as introduced at the beginning of this section:
\[
d(T)=\sum_{p|\;|T|} d(T_{(p)})=\sum_{p|\; |T|} \rank(T_{(p)})
\]
where $T_{(p)}$ denotes the Sylow $p$--subgroup of $T$.

For a finite group $G$, we denote the collection of \textit{proper} abelian subgroups by $\aA(G)$. (It is equal to $\nN(2,G)$ when $G$ is non-abelian.)

\begin{pro}
Let $T$ be a finite abelian group. There is a fibration 
\begin{eqnarray}\label{fibration2}
\bigvee^n S^{d(T)-1} \rightarrow \hocolim{\aA(T)} BA \rightarrow BT
\end{eqnarray}
for some  $n>0$.
\end{pro}
\begin{proof}
The fiber is the coset poset $  \mathcal{C}(T)\simeq\hocolim{\aA(T)} T/A $  and it  is homotopy equivalent to $\bigvee^n S^{d(T)-1}$ by Proposition \ref{cosetposet}.
\end{proof}

The exact sequence of homotopy groups associated to the fibration (\ref{fibration2}) implies the following.

\begin{cor}\label{subcolim}
 When $d(T)\geq 3$, the natural map
\[
\beta\co\colim{\aA(T)} A\rightarrow T 
\]
is an isomorphism.
\end{cor}

\begin{proof}[\textbf{Proof of Theorem \ref{rank2}}]
 We have a sequence of inclusions
\[
\aA_2\rightarrow \aA_3 \rightarrow ...\rightarrow\aA_i\rightarrow \aA_{i+1}\rightarrow ... \rightarrow \aA_r=\aA
\]
for some $r>2$. For each $2\leq i < r$, one can  filter these sets further as follows 
\[
\aA_i=\aA_{i,0} \subset \aA_{i,1}\subset ... \subset \aA_{i,j}\subset... \subset \aA_{i,t_i}=\aA_{i+1} 
\]
where $\aA_{i,j}=\set{B\in \aA_{i+1}|\; \aA(B)\subseteq \aA_{i,j-1}}$ for $0< j \leq t_i $.

For $2\leq i< r$ and $0\leq j<t_i$, each of the maps
\[
\colim{\aA_{i,j}} A\rightarrow \colim{\aA_{i,j+1}} A
\]
has an inverse induced by the compositions: Let $Q\in \aA_{i,j+1}-\aA_{i,j} $ 
\[
Q \stackrel{\beta^{-1}}{\rightarrow} \colim{\aA(Q)} A \rightarrow \colim{\aA_{i,j}} A
\]
where $\beta$ is the map in Corollary \ref{subcolim} and the second map is induced by the inclusion $\aA(Q)\rightarrow \aA_{i,j}$. Hence each inclusion map $\aA_{i,j}\rightarrow \aA_{i,j+1}$ induces an isomorphism on the colimits.
\end{proof}

\section{Colimit of abelian subgroups of extraspecial $2$--groups}\label{sec:colimext}
\label{sec:colim extra}

The colimit of abelian subgroups of a finite group usually turns out to be an infinite group. Simplest examples are amalgamations of the maximal abelian subgroups along the center of the group, see \S \ref{sec:examples} for such examples of groups. The smallest example which diverges from this pattern is the extraspecial groups of order $32$. In this section, we consider colimits of abelian subgroups of extraspecial $2$--groups and prove that it is a finite group. 

 An extraspecial $2$--group of order $2^{2n+1}$  fits into an extension of the form
\[
\ZZ/2 \rightarrow G_n \stackrel{\pi}{\rightarrow} (\ZZ/2)^{2n},
\] 
where $\ZZ/2$ is  the center $Z(G_n)$. There are two, up to isomorphism, extraspecial $2$--groups $G_n^+$ and $G_n^-$ for a fixed $n$. The first one is isomorphic to a central product of $n$ copies of $D_8$ the dihedral group of order $8$, for the second one replace one copy of $D_8$ by $Q_8$ the quaternion group of order $8$.
 
 We want to compute the colimit of abelian subgroups of $G_n$ and we  denote the collection of abelian subgroups by $\aA(G_n)$. Recall that there is a natural surjective map 
\[
\psi\co \colim{\aA(G_n)} A\rightarrow G_n
\]
induced by the inclusions $A\rightarrow G_n$. 

In the case of $G_1$, which is either $Q_8$ or $D_8$, the colimit is an amalgamated product of the maximal abelian subgroups along the center. We consider the case $n\geq 2$.
There is an exact sequence
\begin{equation}\label{eq:quotient}
1\rightarrow Z(G_n)\rightarrow \colim{\aA(G_n)} A \rightarrow \colim{\aA(G_n)} A/Z(G_n)\rightarrow 1
\end{equation}
and a commutative diagram
 \[
\begin{CD}
\colim{\aA(G_n)} A @>\psi>> G_n\\
@VVV @VVV\\
\colim{\aA(G_n)} A/Z(G_n) @>\psi'>> G_n/Z(G_n),
\end{CD} 
\]
where $\psi$ and $\psi'$ are surjective.

It is easier to study $\colim{\aA(G_n)} A/Z(G_n)$, since we can identify this group as a colimit over the polar space of $Sp_{2n}(2)$, see for example \cite{smith} for a discussion of polar spaces. Let us explain this identification.

\subsubsection*{Polar spaces} Regard $(\ZZ/2)^{2n}$ as a vector space over $\ZZ/2=\{ 0,1\}$ and denote it by $V^{2n}$. There is a non-degenerate symplectic bilinear form $V^{2n}\times V^{2n}\rightarrow \ZZ/2$ defined by
\[
(v,w)=[\pi^{-1}(v),\pi^{-1}(w)]
\]
where $[-,-]$ denotes the commutator of given two elements in $G_n$. A subgroup $A$ of $G$ is abelian if and only if $\pi (A)$ is an isotropic subspace of $V^{2n}$ with respect to the bilinear form $(-,-)$, that is $\pi(A)\subset \pi(A)^\bot$. One can order these isotropic subspaces by inclusion. 
\Def{\label{polar space}\rm{ The \textit{polar space of $Sp_{2n}(2)$} is the poset of (non-trivial) isotropic subspaces of $V^{2n}$. We denote this poset by $\sS_{n}$. 
}}
Let $\{ e_i\}_{1\leq i \leq 2n}$ be a basis for $V^{2n}$. It decomposes as the orthogonal direct sum of $2$--dimensional subspaces
\begin{equation}\label{dec}
V^{2n}=\bigoplus_{i=1}^{n} \langle e_{2i-1},e_{2i}\rangle,
\end{equation}
where
\[ 
(e_j,e_k)=\left\lbrace
\begin{array}{cc}
1 & \text{if } \set{j,k}=\set{2i-1,2i}, \\
0 & \text{otherwise}.
\end{array} \right.
\]
For an abelian subgroup $A$ of $G_n$, the assignment $A\mapsto \pi(A)$ defines a surjective map of posets  $\aA(G_n)\rightarrow \sS_n$ which induces an isomorphism
\[
\colim{\aA(G_n)} A/Z(G_n) \rightarrow \colim{\sS_n} S.
\]
Note that we use $\sS_n$  as an indexing category, the colimit is still in \textit{the category of groups}. 

Let us introduce some notation. We denote the multiplication in the image of $\iota_S\co S\rightarrow \colim{\sS_n}S$ by $\iota_S (s_1+s_2)=\iota_S (s_1)\iota_S(s_2)$. In particular, for $S=\Span{e_I}$ where  $e_I=\sum_{i\in I}e_i$ for some non-empty subset $I$ of $ \{1,2,...,2n\}$, we simply write
\[ g_I=\iota_{\Span{e_I}}(e_I) \]
to denote the image. 

Let $\sS_{n,r}$ denote the subposet of $\sS_n$ which consists of subspaces of dimension at most $r$.  Theorem \ref{rank2} implies  that the inclusion $\sS_{n,2}\rightarrow \sS_n$ induces an isomorphism
\[
\colim{\sS_{n,2}}S\rightarrow \colim{\sS_n}S.
\]
 Then a presentation of the colimit can be given as 
\begin{align}\label{presentation}
\colim{\sS_{n}}S =<g_I,\; I\subset \{1,2,...,2n\} |\; g_I^2=1,\;  [g_I,g_J]=1 \;\Leftrightarrow \; (e_I,e_J)=0,\nonumber \\
g_Ig_J=g_K\; \Leftrightarrow \; e_I+e_J=e_K
 >.
\end{align}
Note that this group is generated by the images $g_I=\iota_{\Span{e_I}}(e_I)$ of  $1$--dimensional subspaces  $\Span{e_I}\in \sS_{n,1}$. We will find a smaller collection of subspaces which generates this group. 

For a $1$--dimensional subspace $\langle e_I\rangle\in \sS_{n,1}$ , define
\[
\lL(e_I)=\lbrace W\in \sS_{n,2}|\; \Dim(W)=2 \text{ and } e_I\in W  \rbrace
\]
and  for a sub-collection $\tT\subset \sS_{n,2}$ of $2$--dimensional spaces  let 
\[
\pP(\tT)=\lbrace \Span{w}\in \sS_{n,1}|\;  w\subset X \text{ for some } X\in \tT \rbrace .
\]
It is possible to determine the cardinalities of these sets: $|\lL(e_I)|=2^n-1$ and $|\pP(\lbrace W\rbrace)|=3$ for any $2$--dimensional subspace $W\in \sS_{n,2}$. 

\Def{\rm{
Let $\qQ\subset \sS_{n,1}$ and $ \Span{w}\in \sS_{n,1} $, we say $\Span{w}$ is connected to $\qQ$ if there exists a $W\in \sS_{n,2}$ of dimension $2$ such that $w\in W$ and $\pP(\set{W})-\lbrace \Span{w} \rbrace \subset \qQ$. We also say $\Span{w}\subset W$ is connected to $\qQ$ if the other two elements in  $\pP(\set{W})-\lbrace \Span{w}\rbrace$ are connected to $\qQ$.
}}

\begin{lem}
Consider the poset $\pP_1=\pP(\lL(e_1))\cup \lbrace \langle e_2 \rangle\rbrace$, then every element of $\sS_{n,1}$ is connected to $\pP_1$.
\end{lem}
\begin{proof}
Identify the polar space of $V^{2n}/\langle e_1,e_2\rangle$ with $\sS_{n-1}$ and fix $I=\lbrace 1,2\rbrace$. Then by (\ref{dec}) we have 
\[
\begin{array}{ccc}
\pP(\lL(e_1))&=&\sS_{n-1,1} \sqcup \lbrace \langle e_1\rangle\rbrace\sqcup e_1\sS_{n-1,1},\\
\pP(\lL(e_2))&=&\sS_{n-1,1} \sqcup \lbrace \langle e_2\rangle\rbrace\sqcup e_2\sS_{n-1,1},\\
\pP(\lL(e_I))&=&\sS_{n-1,1} \sqcup \lbrace \langle e_I\rangle\rbrace\sqcup e_I\sS_{n-1,1},
\end{array}
\]
where the notation $w\sS_{n-1,1}$ means the set $\lbrace \langle w+v\rangle|\; \langle v\rangle\in \sS_{n-1,1}\rbrace$. Also observe that
\[
\sS_{n,1}=\pP(\lL(e_1)) \cup \pP(\lL(e_2)) \cup\pP(\lL(e_I)).
\]

Every element $\langle v\rangle$ of $e_2\sS_{n-1,1}$ is connected to $\pP_1$ via the space $\langle e_2,v\rangle$. Therefore every element of $\pP(\lL(e_2))$ is connected to $\pP_1$.

To see that every element of $\pP(\lL(e_I))$ is connected to $\pP_1$, it is sufficient to show that $\langle e_I\rangle$ is connected to $\pP_1$. The sequence of spaces
\[ \langle e_1+e_2,e_3+e_4\rangle\; \hookleftarrow \; \langle e_1+e_2+e_3+e_4\rangle\; \hookrightarrow \; \langle e_1+e_3,e_2+e_4\rangle \]
connects $e_I$ to $\pP_1$.
\end{proof}

Set $K_n=\colim{\sS_{n,2}}S$  and consider the map $\phi\co\colim{\pP_1} S\rightarrow K_n$ induced by the inclusion $\pP_1 \rightarrow \sS_{n,2}$. The lemma translates into the statement that the image of $\phi$ generates 
 $K_n$. This follows from the definition of connectivity: An element $\langle e_I\rangle$ of $\sS_{n,1}$ is connected to $\pP_1$ translates into $g_I$ lies in the image of $\phi$. Hence an equivalent statement is the following.
\begin{lem}\label{generation}
Inclusion of the poset $ \pP_1 \rightarrow \sS_{n,2}$
 induces a surjective map on the colimits
 \[
 \phi\co\colim{\pP_1} S\rightarrow K_n.
 \]
 \end{lem}

By definition, $g_1=\iota_{\Span{e_1}}(e_1)$ commutes with the generators coming from $\pP(\lL(e_1))$.
We will prove that in fact $g_1$ also commutes with $g_2$ the generator corresponding to $e_2$, that is $g_1$ commutes with the generators  $\pP_1=\pP(\lL(e_1))\cup \lbrace \langle e_2 \rangle\rbrace$ of the group $K_n$. Hence $g_1$ is in the center $Z(K_n)$. 

\begin{lem}\label{colim}
For $n\geq 2$
\[
\colim{\aA(G_n)} A/Z(G_n)\cong (\ZZ/2)^{2n+1}. 
\]
\end{lem}
\begin{proof}
First recall from \ref{presentation} that each generator of $K_n$ has order $2$, showing that it is also abelian is sufficient.
We do induction on $n$. 
First, we compute the case $n=2$. 
By the decomposition (\ref{dec}), it is easy to see
\[
\pP(\lL(e_1))=\lbrace \langle e_1\rangle ,\langle e_3\rangle ,\langle e_1+e_3\rangle ,\langle e_4\rangle ,\langle e_1+e_4\rangle ,
\langle e_3+e_4\rangle ,\langle e_1+e_3+e_4\rangle \rbrace.
\]
One can check that the minimal number of generators of the image of $\phi$, in Lemma \ref{generation}, is $5$ and those are $g_1$, $g_3$, $g_4$, $g_{3,4}$ and $g_2$. Observe that $g_1$ commutes with all but $g_2$. We have the following relations in $K_2$

\begin{equation}\label{relations}
\begin{array}{ccc}
g_2g_{2,3,4}&=& g_1g_{1,3,4},\\
g_2g_{1,3}g_{1,2,4}&=&g_1g_{1,3,4},\\
g_2g_{1,3}g_{2,3}g_{1,3,4}&=&g_1g_{1,3,4},\\
g_2g_1g_3g_2g_3&=&g_1,\\
g_2g_1g_2&=&g_1. 
\end{array}
\end{equation}
 Therefore $g_1$ is in the center of $K_2$. It suffices to show that the remaining generators in $J=\set{g_1,g_3,g_4,g_{3,4},g_2}$ are also central. Given $g$ in $J$ there exists an automorphism of $(\ZZ/2)^4$  which induces an automorphism on $K_2$ such that $g$ is mapped to $g_1$. Hence any element in $J$ is central in $K_2$. We conclude that $K_2$ is an abelian group since it is generated by the elements in $J$. 
 
For the general case, first identify the polar space of $V^{2n}/\langle e_1,e_2\rangle$ with $\sS_{n-1}$ and by induction the associated colimit is $K_{n-1}\cong (\ZZ/2)^{2n-1}$. Consider $K_{n-1}$ as a subgroup of $K_n$. By Lemma \ref{generation}, $K_n$ has $2n+1$ generators, namely $e_1$, $e_2$,  and the  $2n-1$ generators of $K_{n-1}$. Note that $g_1$ and $g_2$ centralizes $K_{n-1}$ and the relations   
 (\ref{relations}) hold in $K_n$. Hence $[g_1,g_2]=1$.
\end{proof} 
 Recall that the fundamental group of $B(2,G_n)$ is isomorphic to the colimit of abelian subgroups of $G_n$.  
\begin{thm}\label{colim2}
Let $G_n$ denote an extraspecial $2$--group of order $2^{2n+1}$ then 
\[ \pi_1(B(2,G_n))\cong \colim{\nN(2,G_n)}A\cong G_n\times \ZZ/2 \;\; \text{ for } n\geq2. \]
\end{thm}  
\begin{proof}
Let us denote the colimit in the theorem by $\widehat{K}_n$, recall the exact sequence of (\ref{eq:quotient}). There are exact sequences of groups
\[
\begin{CD}
 @. \ZZ/2 @= \ZZ/2 \\
 @. @VVV    @VVV\\
Z(G_n)@>>> \widehat{K}_n @>\pi>> K_n \\
@| @VV\psi V  @VV\psi'V\\
Z(G_n) @>>> G_n @>>> (\ZZ/2)^{2n} 
\end{CD}
\] 
where $Z(G_n)\cong \ZZ/2$ and by Lemma \ref{colim},  $K_n\cong (\ZZ/2)^{2n+1}$ which forces $\ker(\psi)\cong\ZZ/2$. Note that $Z(G_n)\cap \ker(\psi)=1$ which implies that this kernel is in the center $Z(\widehat{K}_n)$, and $\psi$ induces an isomorphism
\[ [\widehat{K}_n,\widehat{K}_n]\cong [G_n,G_n].\]
This implies that $[\widehat{K}_n,\widehat{K}_n] \cap \ker(\psi)=1$ and $[\widehat{K}_n,\widehat{K}_n]=Z(G_n)$. 

Denote the product $[\widehat{K}_n,\widehat{K}_n]\ker(\psi)$ in $\widehat{K}_n$ by $V$.
Choose a complement $\overline{C}$ of $V/[\widehat{K}_n,\widehat{K}_n]$ in $\widehat{K}_n/[\widehat{K}_n,\widehat{K}_n]$. Let $C$ be the corresponding subgroup of $\widehat{K}_n$. Then $\widehat{K}_n=C\ker(\psi)$ and $C\cap \ker(\psi)=1$, hence
\[
C\cong \widehat{K}_n/\ker(\psi) \cong G_n.
\]
This implies that $\widehat{K}_n$ is isomorphic to the direct product $G_n\times \ZZ/2$.

\end{proof}  

\section{Examples}\label{sec:examples}

\subsection{Transitively commutative groups}
A non-abelian group $G$ is called a \textit{transitively commutative group (TC--group)} if commutation is a transitive relation for non-central elements that is $[x,y]=1=[y,z]$ implies $[x,z]=1$ for all $x,y,z\in G-Z(G)$.  Any pair of maximal abelian subgroups of $G$ intersect at the center $Z(G)$. 

For TC--groups, the homotopy types of $E(2,G)$ and $B(2,G)$ are studied in \cite{adem1}. The exact sequence (\ref{eq:homotopy}) of homotopy groups associated to the fibration $E(2,G)\rightarrow B(2,G)\rightarrow BG$ gives
\[
1\rightarrow T(2)\rightarrow G(2)\rightarrow G\rightarrow 1
\]
where $G(2)$ is isomorphic to the amalgamated product of  maximal abelian subgroups $\{ M_i \}_{1\leq i\leq k}$ along the center of the group, and $T(2)$ is a free group of certain rank.
It turns out that $E(2,G)$ and $B(2,G)$ are $K(\pi,1)$ spaces: $E(2,G)$ is homotopy equivalent to a wedge of circles and $B(2,G)\simeq B(G(2))$.

Higher limits of the representation ring functor $R\co\nN(2,G)\rightarrow \catAb$ fit into an exact sequence
\[
0\rightarrow \ilim{}{0} R\rightarrow R(Z(G))\oplus \bigoplus_{i=1}^k R(M_i) \stackrel{\theta}{\rightarrow}  \bigoplus_{i=1}^k R(Z(G))\rightarrow \ilim{}{1} R\rightarrow 0 
\]
where $\theta(z,m_1,...,m_k)= (z-\res_{M_1,Z(G)}m_1,...,z-\res_{M_k,Z(G)}m_k)$. No higher limit occurs $\ilim{}{1} R=0$, that is $\theta$ is surjective, since
the restriction maps are surjective $R(M_i)\rightarrow R(Z(G))$. The same conclusion holds for the restriction of the functor $R$ to $\nN(2,G)_p\subset \nN(2,G)$.
Therefore the spectral sequence in (\ref{E2}) collapses at the $E_2$--page and after completion
\[
K^i(B(2,G)) 
\cong \left\{
\begin{array}{ll}
\ZZ\oplus \bigoplus_{p\mid |G|} \ZZ_p^{n_p} &  \text{ if $i=0$, }  \\
0 &\text{ if $i=1$.} 
\end{array} \right.
\]

\subsection{Extraspecial $2$--groups}
This class of groups are interesting since they provide counter examples to the following question.
\begin{que}\cite[page 15]{adem1}
\rm{If $G$ is a finite group,  are the spaces $B(q,G)$ Eilenberg--Mac Lane spaces of type $K(\pi,1)$? }
\end{que}
Recall from \S \ref{sec:colim extra} that  an extraspecial $2$--group of order $2^{2n+1}$  is an extension of the form
\[
\ZZ/2 \rightarrow G_n \stackrel{\pi}{\rightarrow} (\ZZ/2)^{2n}
\] 
where $\ZZ/2$ is  the center $Z(G_n)$. We consider the homotopy type of $B(2,G_n)$ for this class of groups. There is a fibration $E(2,G_n)\rightarrow B(2,G_n)\rightarrow BG_n$ and a corresponding exact sequence of homotopy groups  (\ref{eq:homotopy})
\[
1\rightarrow \pi_1(E(2,G_n)) \rightarrow \pi_1(B(2,G_n))\stackrel{\psi}{\rightarrow} \pi_1(BG_n)\rightarrow 1,
\]
where $\psi$ is the natural map
\[
\pi_1(B(2,G_n))\cong \colim{\nN(2,G_n)} A \rightarrow G
\]
induced by inclusions of abelian subgroups $A\subseteq G$. We calculated this colimit in Theorem \ref{colim2}
\[
\colim{\nN(2,G_n)} A\cong G_n\times \ZZ/2
\]
which forces the kernel of $\psi$ to be isomorphic to $\ZZ/2$. This implies that the fundamental group of $E(2,G_n)$ is isomorphic to $\ZZ/2$. Recall the homotopy colimit description of $E(2,G)$, there is a weak equivalence
\[
E(2,G_n)\simeq \hocolim{\nN(2,G_n)} G_n/A.
\]
This homotopy colimit is a finite dimensional complex. It can be described as the nerve of a category with finitely many morphisms. Therefore $H^*(E(2,G_n);\ZZ)$ is non-zero only for finitely many degrees whereas it is well-known that $H^*(B\ZZ/2;\ZZ)$ is non-zero for every even degree. Therefore $E(2,G_n)$ does not have the homotopy type of a $K(\ZZ/2,1)$ space, that is it has non-vanishing higher homotopy groups. The same conclusion holds for $B(2,G_n)$ since from the fibration $E(2,G)\rightarrow B(2,G)\rightarrow BG$ we see that
\[
\pi_k(E(2,G_n))\cong \pi_k(B(2,G_n))\; \text{ for } k\geq 2.
\]
Thus $B(2,G_n)$ does not have the homotopy type of $K(G_n\times \ZZ/2,1)$. 

Indeed for the extraspecial $2$--groups of order $32$, $G_2=G^+_2$ or $G^-_2$, consider the fibration
\[
\begin{CD}
\widetilde{B(2,G_2)} &@>>> & B(2,G_2) \\
  &@. & @VVV \\
  &@. & B(G_2\times \ZZ/2)
\end{CD}
\]
where the fibre is a simply connected  complex of dimension $2$
\[
\widetilde{B(2,G_2)}= \hocolim{\nN(2,G_2)} (G_2\times \ZZ/2)/A.
\]
Cohomology groups of this fiber can be calculated from the chain complex 
\[
 \prod_{A_0\leftarrow A_1\leftarrow A_2} \ZZ[(G_2\times \ZZ/2)/A_2] \rightarrow 
\prod_{A_0\leftarrow A_1} \ZZ[(G_2\times \ZZ/2)/A_1] \rightarrow \prod_{A_0} \ZZ[(G_2\times \ZZ/2)/A_0],
\]
where $A_i$ are abelian subgroups  of $G_n$. It is sufficient to consider the ones which contain the center $Z(G_n)$ (Remark \ref{cofinal}). By a counting argument using the structure of the poset which is isomorphic to the polar space (Definition \ref{polar space}) of $Sp_4(2)$ (union the trivial space $0$), one can calculate the Euler characteristic 
\begin{eqnarray*}
\chi (\widetilde{B(2,G_2)})&=&  \frac{64}{2}|\set{0\subset W_1 \subset W_2}|- \frac{64}{2|W_1|} |\set{ W_1 \subset W_2}|+ \frac{64}{2|W_2|}|\set{  W_2}| \\
&=&  \frac{64}{2}45- \left( \frac{64}{4}45+\frac{64}{2}30 \right) + \left( \frac{64}{8}15+\frac{64}{4}15+ \frac{64}{2} \right) \\
&=& 152,
\end{eqnarray*}
where $W_i$ are the isotropic subspaces of $V^4=(\ZZ/2)^4$ as described in \S \ref{sec:colim extra}.
It follows that the higher homotopy groups of $B(2,G_2)$ are isomorphic to the homotopy groups of a wedge of spheres
\[
\pi_i(B(2,G_2))\cong \pi_i(\bigvee^{151}S^2)\; \text{ for } i>1.
\]

As for the $K$--theory, non-vanishing higher limits of the representation ring functor occur in the Bousfield--Kan spectral sequence. Again in the case of $G_2$ a calculation in GAP \cite{gap} shows that 
\[\ilim{\nN(2,G_2)}{i} R \cong \left\lbrace \begin{array}{ll}
\ZZ^{32} & i=0, \\
(\ZZ /2)^9 & i=1, \\
0 & \text{ otherwise.}
\end{array}  \right. \] 
Therefore the spectral sequence in (\ref{E2}) collapses at the $E_2$--page and after completion
\[
K^i(B(2,G_2)) 
\cong \left\{
\begin{array}{ll}
 \ZZ \oplus \ZZ_2 ^{31}&  \text{ if $i=0$, }  \\
(\ZZ /2)^9 &\text{ if $i=1$.} 
\end{array} \right.
\]
Note that the $K$--theory computation also confirms  that $B(G_2,2)$ is not homotopy equivalent to $B(G_2\times \ZZ/2)$, as the latter has $K^1(B(G_2\times \ZZ/2))=0$ as a consequence of the Atiyah--Segal completion theorem.

\section{Appendix}\label{sec:app}

In this section we review some definitions about simplicial sets especially homotopy colimits. We refer the reader to \cite{goerss} and \cite{dwyer1}.

The category of simplicial sets $\catS$ and the category of topological spaces $\catTop$ are closely related. Given $T$ a topological space there is a functor  $\catTop\rightarrow \catS$ called the \textit{singular set}  defined by
\[
\Sing(T)\co \catn\mapsto \Hom_\catTop(|\Delta^n|,T)
\]
where $|\Delta^n|$ is the topological standard $n$--simplex. Conversely a simplicial set $X$ gives rise to a topological space via the \textit{geometric realization} functor
\[
|X|=\coker \left\lbrace \coprod_{\theta\co \catm\rightarrow\catn} X_n\times |\Delta^m|\rightrightarrows\coprod_{\catn} X_n\times |\Delta^n| \right\rbrace 
\]
where $(x,p)\mapsto (x,\theta_*(p))$ and $(x,p)\mapsto (\theta^*(x),p)$. There is an adjoint relation between these two functors
\[
|.|\co \catS \rightleftarrows \catTop :\!\Sing.
\]
Note that $|.|$ preserves all colimits in $\catS$ since it has a right adjoint.

Let $\catI$ be a small category, and $i\in \catI$ an object,  the under category $\catI\setminus i$ is the category whose objects are the maps $i\rightarrow i_0$ and  
 morphisms  from $i\rightarrow i_0$ to $i\rightarrow i_1$ are commutative diagrams
 \[
 \begin{diagram}
 i & \rTo & i_0 \\
  &\rdTo & \dTo \\
  &&i_1. 
 \end{diagram}
 \]
Let $F\co\catI \rightarrow \catTop$ be a functor. The \textit{homotopy colimit} of $F$ is the topological space
\[
\hocolim{} F=\coker \left\lbrace \coprod_{f\co j \rightarrow i} B(\catI\setminus i)  \times F(j)\rightrightarrows\coprod_{i} B(\catI\setminus i)\times F(i) \right\rbrace 
\]
where $(b,y)\mapsto (B(f)(b),y)$ and $(b,y)\mapsto (b,f(y))$. If one wants to work in the category of simplicial sets, given $F\co \catI \rightarrow \catS$ one can use the simplicial set $B_\n(\catI\setminus i)$ instead. 

From the definition we see that there is a natural map to the ordinary colimit
\[
\hocolim{} F\rightarrow \colim{} F
\]
obtained by collapsing the nerves $B(\catI\setminus i)\rightarrow \pt$. This map is not usually a weak equivalence. In general, it is a weak equivalence if the diagram $F\co \catI\rightarrow \catS$ is a \textit{free} diagram (\cite[Appendix HC]{far2}). A diagram of sets $\catI\rightarrow \catSet$ is free if it is of the form $\coprod_i F^i$ for a collection of subobjects of $\catI$ where each diagram $F^i\co \catI\rightarrow \catSet$  is defined by $F^i(j)=\Hom_\catI(i,j)$. A diagram of simplicial sets is free if in each dimension it gives a free diagram of sets. In particular, a collection of sets which is closed under taking subsets and partially ordered by inclusion gives a free diagram of sets.

 A standard example of a homotopy colimit is the Borel construction. Let $X$ be a $G$--space and consider $\catG$ the category associated to the group then   
\[
\hocolim{} F \simeq  X\times_G EG 
\]
where $F\co \catG\rightarrow \catTop$ sends the single object to $X$.
 The most important property of homotopy colimits is that  a natural transformation $F\rightarrow F'$ of functors such that $F(i)\rightarrow F'(i)$ is a weak equivalence for all $i\in \catI$ induces a weak equivalence
\[
\hocolim{} F\rightarrow \hocolim{} F'.
\]
 Another useful property we use in the paper is the commutativity of homotopy colimits. We refer the reader to \cite{wzz} for further properties of homotopy colimits.


\begin{thebibliography}{9}

\bibitem{adem1} \textsc{A. Adem, F.R. Cohen, E. Torres-Giese}:
\emph{Commuting elements, simplicial spaces and filtrations of classifying spaces}, Math. Proc. Cam. Phil. Soc. Vol. 152 (2011), 91-114.

\bibitem{a-s} \textsc{M. F. Atiyah, G.B. Segal}:
\emph{Equivariant K-Theory and Completion}, J. Differential Geom. Volume 3, Number 1-2 (1969), 1-18.

\bibitem{B-K} \textsc{A.K. Bousfield, D.M. Kan}:
\emph{Homotopy Limits, Completions and Localizations}, Lecture Notes in Mathematics, Vol. 304. Springer-Verlag, Berlin-New York, 1972. 

\bibitem{brown} \textsc{K. Brown}:
\emph{The coset poset and probabilistic zeta function of a finite group},  J. Algebra 225 (2000), no. 2, 989-1012. 


\bibitem{D-S-M-S} \textsc{J.D. Dixon, M.P.F. Du Sautoy, A. Mann, D. Segal}:
\emph{Analytic Pro-$p$ Groups}, Cambridge Studies in Advanced Mathematics, 61. Cambridge University Press, Cambridge, 1999.


\bibitem{dress} \textsc{A. W. M. Dress}:
\emph{Contributions to the theory of induced representations}, Algebraic K-theory, II: "Classical'' algebraic K-theory and connections with arithmetic (Proc. Conf., Battelle Memorial Inst., Seattle, Wash., 1972), pp. 183-240. Lecture Notes in Math., Vol. 342, Springer, Berlin, 1973. 

\bibitem{dwyer1} \textsc{W. G. Dwyer}:
\emph{Classifying spaces and homology decompositions}, Topology 36 (1997), no. 4, 783-804. 

\bibitem{far} \textsc{E. D.  Farjoun}:
\emph{Fundamental group of homotopy colimits}, Adv. Math. 182 (2004), no. 1, 1–27. 

\bibitem{far2} \textsc{E. D.  Farjoun}:
\emph{Cellular spaces, Nullspaces and Homotopy localization}, Springer LNM 1622 (1996)

\bibitem{gap}
 \textsc{ The GAP~Group}: 
  \emph{GAP -- Groups, Algorithms, and Programming, 
  Version 4.4.12}, 
  2008,
  \verb+(http://www.gap-system.org)+.

\bibitem{goerss} \textsc{P. G. Goerss, J. F. Jardine}:
\emph{Simplicial homotopy theory}, Progress in Mathematics, 174. Birkhäuser Verlag, Basel, 1999. 

\bibitem{grodal} \textsc{J. Grodal}:
\emph{Higher limits via subgroup complexes}, Ann. of Math. (2) 155 (2002), no. 2, 405-457.

\bibitem{JM} \textsc{S. Jackowski, J. McClure}:
\emph{Homotopy decomposition of classifying spaces via elementary abelian subgroups.}, Topology 31 (1992), no. 1, 113–132. 

\bibitem{lee} \textsc{Lee, Chun-Nip}:
\emph{A homotopy decomposition for the classifying space of virtually torsion-free groups and applications.}, Math. Proc. Cambridge Philos. Soc. 120 (1996), no. 4, 663-686. 

\bibitem{luck} \textsc{W. L\"{u}ck}:
\emph{Rational computations of the topological K-theory of classifying spaces of discrete groups} (English summary), J. Reine Angew. Math. 611 (2007), 163-187. 


\bibitem{N-S} \textsc{N. Nikolov, D. Segal}:
\emph{Generators and commutators in finite groups; abstract quotients of compact groups. } (English summary), Invent. Math. 190 (2012), no. 3, 513-602. 

\bibitem{oliver} \textsc{B. Oliver}:
\emph{Higher Limits Via Steinberg Representations}, Comm. Algebra 22 (1994), no. 4, 138-1393. 

\bibitem{segal} \textsc{G. Segal}:
\emph{Equivariant K-Theory}, Inst. Hautes Études Sci. Publ. Math. No. 34 1968 129-151. 

\bibitem{smith} \textsc{S. D. Smith}:
\emph{Subgroup complexes} Mathematical Surveys and Monographs, 179. American Mathematical Society, Providence, RI, 2011.

\bibitem{webb} \textsc{P. Webb}:
\emph{An introduction to the representations and cohomology of categories}, Group representation theory, 149-173, EPFL Press, Lausanne, 2007. 


\bibitem{wzz} \textsc{V. Welker, G. Ziegler, R. T. {\v{Z}}ivaljevi{\'c}}:
\emph{Homotopy colimits---comparison lemmas for combinatorial applications}, J. Reine Angew. Math. 509 (1999), 117-149.
\end{thebibliography}
\end{document}